\documentclass[11pt,reqno]{amsart}

\usepackage[T1]{fontenc}
\usepackage[a4paper,margin=1.15in]{geometry}
\usepackage{amsmath,amssymb,amsthm,mathtools}
\usepackage{enumitem}
\usepackage{microtype}
\usepackage{xcolor}
\usepackage{aliascnt}
\usepackage[colorlinks=true,linkcolor=blue!50!black,citecolor=blue!50!black,urlcolor=blue!50!black]{hyperref}
\usepackage[nameinlink,noabbrev,capitalise]{cleveref}
\numberwithin{equation}{section}

\theoremstyle{plain}
\newtheorem{theorem}{Theorem}[section]
\crefname{theorem}{Theorem}{Theorems}
\Crefname{theorem}{Theorem}{Theorems}

\newaliascnt{observation}{theorem}
\newtheorem{observation}[observation]{Observation}
\aliascntresetthe{observation}
\crefname{observation}{Observation}{Observations}
\Crefname{observation}{Observation}{Observations}

\newaliascnt{lemma}{theorem}
\newtheorem{lemma}[lemma]{Lemma}
\aliascntresetthe{lemma}
\crefname{lemma}{Lemma}{Lemmas}
\Crefname{lemma}{Lemma}{Lemmas}

\newaliascnt{proposition}{theorem}
\newtheorem{proposition}[proposition]{Proposition}
\aliascntresetthe{proposition}
\crefname{proposition}{Proposition}{Propositions}
\Crefname{proposition}{Proposition}{Propositions}

\newaliascnt{conjecture}{theorem}
\newtheorem{conjecture}[conjecture]{Conjecture}
\aliascntresetthe{conjecture}
\crefname{conjecture}{Conjecture}{Conjectures}
\Crefname{conjecture}{Conjecture}{Conjectures}

\theoremstyle{definition}
\newaliascnt{definition}{theorem}
\newtheorem{definition}[definition]{Definition}
\aliascntresetthe{definition}
\crefname{definition}{Definition}{Definitions}
\Crefname{definition}{Definition}{Definitions}

\crefname{section}{Section}{Sections}
\Crefname{section}{Section}{Sections}
\crefname{subsection}{Subsection}{Subsections}
\Crefname{subsection}{Subsection}{Subsections}
\crefname{equation}{Equation}{Equations}
\Crefname{equation}{Equation}{Equations}

\newcommand{\calB}{\mathcal B}
\newcommand{\calF}{\mathcal F}
\newcommand{\calH}{\mathcal H}
\newcommand{\calM}{\mathcal M}
\newcommand{\calQ}{\mathcal Q}
\newcommand{\calR}{\mathcal R}
\newcommand{\ZZ}{\mathbb Z}
\newcommand{\eps}{\varepsilon}
\DeclareMathOperator{\supp}{supp}
\DeclareMathOperator{\ex}{ex}

\begin{document}

\title[\texorpdfstring{$\ell$}{ell}-degree Tur\'an densities of hypergraphs]{Rational values and non-principality for \texorpdfstring{$\ell$}{ell}-degree Tur\'an densities of hypergraphs}
	
\author{Quanyu Tang}
\address{School of Mathematics and Statistics, Xi'an Jiaotong University, Xi'an, China}
\email{tangquanyu827@gmail.com}
	
\author{Wenling Zhou}
\address{School of Mathematics, Shandong University, Jinan, China}
\email{gracezhou@sdu.edu.cn}

\begin{abstract}
Let $k>\ell\ge 1$ be integers. For a family $\mathcal{F}$ of $k$-uniform hypergraphs ($k$-graphs), the \emph{$\ell$-degree Tur\'an density} $\gamma^{(k)}_\ell(\mathcal{F})$ of $\mathcal{F}$ is defined as the asymptotic maximum of the normalized minimum $\ell$-degree over all $\mathcal{F}$-free $k$-graphs.

In this paper, we prove that for all $k>\ell>k/2$, every rational number $\alpha\in[0,1)$ can be realized as the $\ell$-degree Tur\'an density $\gamma^{(k)}_\ell(\mathcal{F})$ for some finite family $\mathcal{F}$ of $k$-graphs. Furthermore, for any $k>\ell\ge 1$, we construct an explicit infinite sequence of values realized by single forbidden $k$-graphs: for each integer $q\ge2$, there exists a $k$-graph $F$ such that $\gamma^{(k)}_\ell(F)=1-q^{\ell-k}$.
We also establish a strengthened non-principality property for $\ell$-degree Tur\'an densities: for all  $k>\ell>1$, there exist two $k$-graphs $F_1$ and $F_2$ satisfying $0< \gamma^{(k)}_{\ell}(\{F_1,F_2\})<
\min\{ \gamma^{(k)}_\ell(F_1),\gamma^{(k)}_\ell(F_2)\}$.
\end{abstract}


\maketitle

\section{Introduction}\label{sec:intro}
	
	Extremal graph theory asks how dense a large discrete structure can be when a prescribed configuration is forbidden.
	For an integer $k\ge 2$, a \emph{$k$-uniform hypergraph}, or simply a \emph{$k$-graph}, is a pair $H=(V,E)$ with $E\subseteq \binom{V}{k}$.  Given a (possibly infinite) family $\calF$ of $k$-graphs, the classical Tur\'an number $\ex(n,\calF)$ is the maximum number of edges in an $n$-vertex $\calF$-free $k$-graph, where $\calF$-free means containing no member of $\calF$ as a subgraph.  A standard averaging argument shows that the sequence $\ex(n, \mathcal{F})/\binom{n}{k}$ is non-increasing in $n$.
	Therefore, one often focuses on the \emph{Tur\'an density} of $\mathcal{F}$, defined by
	\[
	\pi(\mathcal{F})=\lim_{n\to \infty}{\ex(n,\mathcal{F})}/{\binom{n}{k}}.
	\]
	If $\mathcal F=\{F\}$, we write $\ex(n,F)$ and $\pi(F)$.
	Throughout the paper, forbidden families are assumed to be nonempty and forbidden $k$-graphs are assumed to have at least one edge; allowing the empty family adds only the trivial value $1$.

	Let $\Pi^{(k)}$, $\Pi^{(k)}_{\rm fin}$ and $\Pi^{(k)}_{\infty}$ denote, respectively, the sets of Tur\'an densities of single $k$-graphs, finite families of $k$-graphs and arbitrary nonempty families of $k$-graphs.
	Clearly, for every $k\ge 2$, $\Pi^{(k)} \subseteq \Pi^{(k)}_{\rm fin}\subseteq \Pi^{(k)}_{\infty}$.
	For the graph case $k=2$, Erd\H{o}s--Stone--Simonovits Theorem~\cite{ErdosStone1946,ErdosSimonovits1966} determines the Tur\'an density of every graph in terms of its chromatic number, and consequently
	\begin{equation}\label{eq:classic-turan-density}
		\Pi^{(2)}_{\infty}=\Pi^{(2)}_{\rm fin}=\Pi^{(2)} =\{(r-1)/r: r\in \mathbb N^+\}. 
	\end{equation}
	For $k\ge 3$, the problem becomes considerably harder and, despite much
	effort, remains wide open even for $3$-graphs. Determining the Tur\'an density even for
	seemingly ``simple'' $3$-graphs is notoriously difficult; see the surveys of Keevash~\cite{Keevash2011} and Sidorenko~\cite{Sidorenko1995}. 
	
Hypergraphs also exhibit phenomena which are absent for graphs.  
Frankl and R\"odl~\cite{FranklRodl1984} disproved a \$1000 conjecture of Erd\H{o}s by showing that, for $k\ge3$, the set $\Pi^{(k)}_\infty$ is not well-ordered.
Pikhurko \cite{Pikhurko2012} later established a collection of fundamental results regarding the sets $\Pi^{(k)}_{\infty}$ and $\Pi^{(k)}_{\rm fin}$.
	In particular, he showed that~$\Pi^{(k)}_{\infty} \neq \Pi^{(k)}_{\rm fin}$, the set $\Pi^{(k)}_{\infty}$ is uncountable, and $\Pi^{(k)}_{\infty} = \overline{\Pi^{(k)}_{\rm fin}}$, where the final identity relies on an earlier result of Brown and Simonovits \cite{BS:84}.
	However, a complete characterization of the sets $\Pi^{(k)}_{\infty}$ and $\Pi^{(k)}_{\rm fin}$ remains an outstanding open problem. In addition,
	Mubayi and R\"odl~\cite{MubayiRodl2002} conjectured that there exist \emph{non-principal} families $\mathcal F$ of $k$-graphs for $k\ge 3$, that is, families $\mathcal F$ such that
	\[
	\pi(\mathcal F)<\min\{\pi(F):F\in\mathcal F\},
	\]
	and remarked that this should hold even when $|\mathcal F|=2$. 
	Balogh~\cite{Balogh2002} confirmed this conjecture by constructing a finite non-principal family of $3$-graphs whose size is much larger than two. 
	Subsequently, Mubayi and Pikhurko \cite{MubayiPikhurko2008} improved this result by constructing size-two non-principal $k$-graph families for all integers $k\ge 3$.
	These examples illustrate the difference between graphs and hypergraphs, motivating extensive research on hypergraph Tur\'an-type problems~\cite{BaloghClemenLidicky2022,ErdosSos1982,Reiher2020,MubayiZhao2007,LoMarkstrom2014,HalfpapLemonsPalmer2025}.
	
	\subsection{\texorpdfstring{$\ell$}{ell}-degree Tur\'an density}
A natural refinement of the classical Tur\'an density  replaces the global edge-density condition with a local degree condition. To formulate this, we recall the standard notions of links and degrees in $k$-graphs.
	Let $H$ be a $k$-graph and let $0< \ell<k$.  For $T \in\binom{V(H)}{\ell}$, its \emph{link} is
	\[
	N_H(T)=\left\{T'\in\binom{V(H)\setminus T}{k-\ell}: T\cup T'\in E(H)\right\},
	\]
	and its \emph{degree} is $d_H(T)=|N_H(T)|$. The \emph{minimum $\ell$-degree} of $H$ is $\delta_\ell(H)=\min\{d_H(T):T\in\binom{V(H)}\ell\}$.
In the special case $\ell=k-1$, $\delta_{k-1}(H)$ is referred to as the \emph{minimum codegree} of $H$.
	For a family $\calF$ of $k$-graphs, the \emph{$\ell$-degree Tur\'an number} $\ex^{(k)}_\ell(n,\calF)$ is the maximum $\delta_\ell(H)$ that an
	$n$-vertex $\calF$-free $k$-graph $H$ can admit, 
	and the \emph{$\ell$-degree Tur\'an density} is defined as
	\begin{equation*}
\gamma^{(k)}_\ell(\calF)=\lim_{n\to\infty}\frac{\ex^{(k)}_\ell(n,\calF)}{\binom{n-\ell}{k-\ell}}.
	\end{equation*}
The limit exists by a result of Lo and Markstr\"om~\cite{LoMarkstrom2014}.
For each family $\calF$ of $k$-graphs, we have
\begin{equation*}
\gamma^{(k)}_{k-1}(\calF)\le \gamma^{(k)}_{k-2}(\calF)\le \dots \le \gamma^{(k)}_{1}(\calF)=\pi(\calF).
\end{equation*}
Thus, the $\ell$-degree Tur\'an density generalizes both the codegree density and the classical Tur\'an density. The codegree setting was introduced by Mubayi and Zhao~\cite{MubayiZhao2007}, while the general $\ell$-degree version was introduced by Lo and Markstr\"om~\cite{LoMarkstrom2014}.
	
	Similarly to the classical Tur\'an density, let $\Gamma^{k,\ell}_{\infty}$, $\Gamma^{k,\ell}_{{\rm fin}}$ and $\Gamma^{k,\ell}$ denote the sets of all possible $\ell$-degree Tur\'an densities of arbitrary families, finite families, and single $k$-graphs, respectively. Then for all $0<\ell<k$, 
	\begin{equation*}
		\Gamma^{k,\ell} \subseteq \Gamma^{k,\ell}_{{\rm fin}}\subseteq \Gamma^{k,\ell}_{\infty}\subseteq[0,1).
	\end{equation*}
For $k=2$ this again gives the set in~\eqref{eq:classic-turan-density}.  By contrast, Mubayi and Zhao~\cite{MubayiZhao2007} proved that $\Gamma^{k,k-1}_\infty$ is dense in $[0,1)$ for every $k\ge3$.
Subsequently, Lo and Markstr\"om~\cite{LoMarkstrom2014} proved the analogous density theorem for all $k>\ell>1$, and they further posed the open question of whether $\Gamma^{k, \ell}_{\infty}=[0,1)$.

	Since there are only countably many finite forbidden families, a finite-family realization theorem can hold only for countably many values. The natural first target is therefore the rational values. 
	Recently, Gao, Pikhurko, Rong and Sun~\cite{GaoPikhurkoRongSun2026} proved that every rational number in $[0,1)$ is the codegree Tur\'an density of a finite family of $k$-graphs.  
	Our first theorem extends this rational-value phenomenon from codegree to all minimum $\ell$-degrees above the half-uniformity threshold.
	
	\begin{theorem}
		\label{thm:rational-main}
		Let $k\ge 3$ and let $\ell$ satisfy ${k}/{2}<\ell<k$.
Then every rational number $\alpha\in[0,1)$ belongs to $\Gamma^{k,\ell}_{\rm fin}$.
	\end{theorem}

It is natural to ask whether every rational value in $[0,1)$ can be
realized as the $\ell$-degree Tur\'an density of a single forbidden
$k$-graph.  In the codegree case $\ell=k-1$, Keevash and
Zhao~\cite[Theorem~1.3]{KeevashZhao2007} proved that, for every
integer $q\ge2$, there exists a $k$-graph whose codegree Tur\'an density
is $1-1/q$.  Gao, Pikhurko, Rong and Sun~\cite[Theorem~3.2]
{GaoPikhurkoRongSun2026} recently gave a new proof of this result.
Our next theorem extends this sequence of single-forbidden density
values from the codegree setting to all pairs $k>\ell\ge1$, via an
explicit recursive construction.

\begin{theorem}
\label{thm:q-color-main}
Let $k>\ell\ge 1$ and let $q\ge 1$ be an integer.  Then there exists a $k$-graph $F^{(k)}_{\ell, q}$ such that
		\[
		\gamma^{(k)}_\ell(F^{(k)}_{\ell, q})=1-q^{\ell-k}.
		\]
\end{theorem}

For $\ell=1$, this gives the classical Tur\'an density $\pi(F^{(k)}_{1,q})=1-q^{1-k}$.
In particular, $F^{(2)}_{1,q}=K_{q+1}$ is the complete graph on $q+1$ vertices. In fact, $F^{(k)}_{\ell,q}$ denotes an explicitly constructed $k$-graph
defined recursively (see \cref{def:Fq}). 
Independently, Ai, Ding, Liu
and Yang~\cite{AiDingLiuYang2026} obtained the same sequence of
single-forbidden density values using tree suspensions and transfer
functions. Maybe somewhat surprisingly, their construction and proof are different from ours.

As an application of the family $F^{(k)}_{\ell,q}$, we establish a strengthened non-principality property for $\ell$-degree Tur\'an densities. 
Recall that for classical Tur\'an densities, a family $\mathcal{F}$ is said to be \emph{non-principal} if $\pi(\mathcal{F})<\min_{F\in\mathcal{F}}\pi(F)$. As mentioned above, no non-principal family exists for graphs, whereas such families have been constructed for all $k\ge 3$. 
In the codegree setting, Mubayi and Zhao~\cite{MubayiZhao2007} constructed finite non-principal families and posed the problem of constructing a non-principal family of size two. This problem was recently resolved by Gao, Pikhurko, Rong and Sun~\cite{GaoPikhurkoRongSun2026} for the codegree Tur\'an density. For all integers $k>\ell>1$, we provide an explicit non-principal family $\mathcal{F}$ containing exactly two  $k$-graphs, one of which is the two-color forcing hypergraph $F^{(k)}_{\ell,2}$.
	
\begin{theorem}
\label{thm:nonprincipal-main}
For all integers $k>\ell>1$, there exist two $k$-graphs $F_1$ and $F_2$ such that
\[
0< \gamma^{(k)}_\ell(\{F_1,F_2\})
<
\min\{\gamma^{(k)}_\ell(F_1),\gamma^{(k)}_\ell(F_2)\}.
\]
\end{theorem}
	
\subsection*{Proof overview and organization}
The paper has two complementary parts. In \cref{sec:rational}, we prove \cref{thm:rational-main}.  Write
$\alpha=a/b$ and $r=k-\ell$.  For the upper bound, a sampling argument
produces a bounded-size induced subgraph with almost the same normalized
minimum $\ell$-degree.  We then force one of finitely many extensions of
this base $k$-graph by adding pairwise disjoint new $r$-sets to prescribed
$b$-element families of $\ell$-sets.  For the matching lower bound, we
construct a vertex-colored $k$-graph with color set $\mathbb Z_Q$.  The
color-class proportions are chosen so that the minimum of $r$ independent
samples from the color distribution is uniform on $\mathbb Z_Q$.  A
cyclic interval lemma then shows that the construction contains none of
the forbidden extensions.  The nonuniform color distribution and the cyclic obstruction are the main additional ingredients needed to adapt the codegree argument to the range $\ell>k/2$.

The second part begins in \cref{sec:qcolor}, where we construct the recursive color-forcing $k$-graphs $F^{(k)}_{\ell,q}$ and prove \cref{thm:q-color-main}. Non-$q$-colorability and the balanced $q$-partite construction give the lower bound, while the upper bound follows from an induction on $q+k-\ell$ using an auxiliary $(k-1)$-graph and a lifting argument.
In \cref{sec:stability}, we consider the case $q=2$ and show that every
near-extremal $F^{(k)}_{\ell,2}$-free $k$-graph contains two disjoint
independent sets of size $(1/2-o(1))n$.
Finally, in \cref{sec:nonprincipal}, a random-tournament construction gives the required lower bound for $K^k(t,t)$ and a positive lower bound for the joint density.  For the strict upper bound, the stability lemma, a count of missing crossing edges, and a final random choice force a copy of $K^k(t,t)$, yielding the strict two-element non-principality asserted in \cref{thm:nonprincipal-main}.  We conclude with open problems in \cref{sec:concluding-remarks}.

Throughout the paper, all hypergraphs are finite and simple.  For a positive integer $m$, we write $[m]=\{1,\ldots,m\}$.  We use the standard hierarchy notation $0<a\ll b\ll c$ to mean that the constants are chosen from right to left, each sufficiently small in terms of the constants to its right and the fixed integer parameters.

\section{Rational values above the half-uniformity threshold}\label{sec:rational}
	
	In this section, we prove \cref{thm:rational-main}. 
	The case $\alpha=0$ follows trivially by forbidding a single $k$-edge. Therefore, throughout this section we fix a rational number $\alpha\in (0,1)$ and write  $\alpha=a/b$
	with $0<a<b$ and $\gcd(a,b)=1$.
	
	Let
	\[
	r:=k-\ell,
	\qquad
	\beta:=1-\alpha=\frac{b-a}{b}.
	\]
	The condition $\ell>k/2$ is equivalent to $r<\ell$.  This inequality
	will only be used in the lower-bound construction, where it guarantees
	that a $k$-set has at most one color appearing at least $\ell$
	times.
	
	We shall construct a finite family $\calF_\alpha$ of $k$-graphs by
	extending bounded $m$-vertex $k$-graphs such that $\gamma^{(k)}_\ell(\calF_\alpha)=\alpha$.  The proof has two parts.  
	The upper bound is forced
	by a finite extension family: any $k$-graph whose minimum
	$\ell$-degree is slightly larger than $\alpha$ contains at least one of these
	extensions. 
	For the lower bound, for all sufficiently large integers $n$, we construct a
	vertex-colored $k$-graph $G_{\alpha,n}$ on $n$ vertices.  The vertex set
	is colored according to a non-uniform distribution over $\mathbb{Z}_Q$, and
	whether a $k$-set is an edge is decided by the colors of its vertices
	through the cyclic intervals $I_c$.  This construction has normalized
	minimum $\ell$-degree asymptotic to $\alpha$, while all
	forbidden extensions are excluded by a cyclic interval obstruction.
	
	\subsection{The cyclic obstruction}
	\label{subsec:cyclic}
	
	We first introduce some elementary cyclic facts that will be used in
	the lower-bound construction.
	
	Choose a positive integer $M$ and define $Q:=bM$ and $h:=(b-a)M=\beta Q$.
	We identify $\ZZ_Q$ with $\{0,1,\ldots,Q-1\}$ with addition modulo $Q$. For $x\in\ZZ_Q$, let
	\[
	I_x:=\{x, x+1,\ldots, x+h-1\}\subseteq\ZZ_Q
	\]
	be the cyclic interval of length $h$.  
	Thus, every complete residue class modulo $M$ has size $b$, i.e., $| \{x\in \mathbb{Z}_Q: x \equiv i \pmod M\}|=b$ and every
	interval $I_c$ has relative size $h/Q=\beta$.
	
	We first prove a cyclic
	interval lemma.  It will be used
	later to force a complete residue class among the colors appearing many
	times in the base $k$-graph.  The final assertion gives the counting property of
	such a residue class that will be used to rule out the forbidden
	extensions.
	\begin{lemma}
		\label{lem:cyclic}
		Let $X\subseteq\ZZ_Q$ be nonempty.  Then the following two properties
		are equivalent.
		\begin{enumerate}[label=\textup{(\roman*)}]
			\item There exists a probability distribution $\nu$ supported on $X$
			such that $\nu(I_c)\le \beta$ for every $c\in X$.
			\item There exists  $i\in  \mathbb{Z}_M$ such that $\{x\in \mathbb{Z}_Q: x \equiv i \pmod M\}\subseteq X$.
		\end{enumerate}
Moreover, if $R$ is a complete residue class modulo $M$, then $|I_x\cap R|=b-a$ for every $x\in\ZZ_Q$.
	\end{lemma}
	
	\begin{proof}
		Assume first that $\nu$ is supported on $X$ such that $\nu(I_c)\le \beta$ for every $c\in X$. 
		We first show that $\nu(I_i)\le \beta$ for every $i\in\ZZ_Q$.  If not, choose $j$ with $\nu(I_j)>\beta$.  Moving cyclically forward from $j$, let $j+t$ be the first point of $\supp\nu$. Then for each $0\le u<t$ we have 
		\[
		\nu(I_{j+u+1})=\nu(I_{j+u}) -\nu(j+u)+\nu(j+u+h)\ge \nu(I_{j+u}).
		\]
		Thus, $\nu(I_{j+t})>\beta$, contradicting the assumption since $j+t\in X$.
		Secondly, $\frac1Q\sum_{i\in\ZZ_Q} \nu(I_i) =h/Q=\beta$, since each point of $\ZZ_Q$ is counted
		in exactly $h$ of the intervals $I_i$. Thus,
		we must have $\nu(I_i) =\beta$ for all $i\in\ZZ_Q$, which implies that $0=\nu(I_{i+1})-\nu(I_i)=\nu(i+h)-\nu(i)$ for all $i\in\ZZ_Q$.
		Therefore, $\nu(i+h)=\nu(i)$ for every $i\in\ZZ_Q$.
		
		Note that $\{i+th \pmod Q:t\in\ZZ\} \subseteq \{x\in \mathbb{Z}_Q: x \equiv i \pmod M\}$, since $h=(b-a)M$. 
		Moreover, since $h/M=b-a$ is coprime to $Q/M=b$, repeated addition of
		$h$ runs through exactly one complete residue class modulo $M$. 
		Choose $i\in\supp\nu$.  Then $\nu$ is
		positive on the whole residue class of $i$, and since $\nu$ is
		supported on $X$, this complete residue class is contained in $X$.

		Conversely, suppose that $X$ contains a complete residue class
		$R=\{x\in \mathbb{Z}_Q: x \equiv i \pmod M\}$ for some $i\in  \mathbb{Z}_M$.  Let $\nu$ be the uniform distribution on $R$.  Every cyclic interval of length $h=(b-a)M$ contains exactly $b-a$
		points of $R$, while $|R|=b$.  Hence
		\[
		\nu(I_c)=\frac{b-a}{b}=\beta
		\]
		for every $c\in\ZZ_Q$, proving the equivalence. 
		Finally, every cyclic interval of length $h$ contains exactly $b-a$ points of $R$, since $h=(b-a)M$.
	\end{proof}
	
	We shall use the previous lemma in a robust form.  Since $\ZZ_Q$ is
	finite, if $X$ contains no complete residue class modulo $M$, then
	the inequality in \Cref{lem:cyclic} fails by a positive amount depending
	only on $Q$ and $h$.
	
	\begin{lemma}\label{lem:rho}
		For fixed $Q$ and $h$, there is a constant $\rho=\rho(Q,h)>0$ such that whenever a nonempty set $X\subseteq\ZZ_Q$ contains no complete residue class modulo $M$, every probability distribution $\nu$ supported on $X$ satisfies
		\[
		\max_{x\in X}\nu(I_x)\ge \beta+\rho.
		\]
	\end{lemma}
	
	\begin{proof}
		There are only finitely many subsets of $\ZZ_Q$.  For a fixed nonempty $X\subseteq\ZZ_Q$, the simplex of probability distributions supported on $X$ is compact, and the function $\nu\mapsto \max_{x\in X}\nu(I_x)$ is continuous.  If $X$ contains no complete residue class modulo
		$M$, then \Cref{lem:cyclic} implies that this function is always
		strictly larger than $\beta$. Taking the minimum of these positive gaps over
		all nonempty $X\subseteq\ZZ_Q$ which contain no complete residue class
		modulo $M$ gives the desired constant $\rho>0$.
	\end{proof}

\subsection{The forbidden family}
\label{subsec:forbidden-family}
We now choose the parameters and define the finite forbidden family $\calF_\alpha$.  To construct this family, we also require an Inheritance Lemma which states that minimum $\ell$-degrees are approximately preserved under taking typical induced subgraphs of constant order. Such results have appeared in various forms in the literature and can be proved via a standard concentration analysis; see, for example,~\cite[Lemma 4.9]{Langliling23}.

	\begin{lemma}
		\label{lem:sampling}
		Let $k>\ell>0$ be integers.  For every $\eta>0$ there is $m_0\in \mathbb N$ such that the following holds for all $m\ge m_0$ and all $n\ge m$.  If $H$ is an $n$-vertex $k$-graph satisfying $\delta_\ell(H)\ge p\binom{n-\ell}{k-\ell}$
		for some $p\in[0,1]$, then there is an $m$-set $U\subseteq V(H)$ such that
		\[
		\delta_\ell(H[U])\ge (p-\eta)\binom{m-\ell}{k-\ell}.
		\]
	\end{lemma}

In this section, we choose the integer $M$ introduced in the previous subsection sufficiently
	large so that, with $Q=bM$,
	\begin{equation}\label{eq:Q-large}
		r\ell Q^{-1/r}<\frac{\beta}{4}.    
	\end{equation}
	Let $\rho=\rho(Q,h)$ be the constant from \cref{lem:rho}.  Choose
	$\eta>0$ so that $ \eta<\min\{\alpha/10,\rho/10\}$.
	Finally, choose $m$ sufficiently large such that \cref{lem:sampling}
	applies with error $\eta$, and
	\begin{equation}\label{eq:m-choice}
		m>rQ(\ell-1),\qquad \binom m\ell\ge b,\qquad
		\frac{\beta+\eta+r\ell/m}{1-rQ(\ell-1)/m}<\beta+\rho.    
	\end{equation}
	The last condition is possible because the left-hand side tends to
	$\beta+\eta$ as $m\to\infty$.
	
	Let $\calH_m$ be a set containing one representative from each
	isomorphism class of $m$-vertex $k$-graphs $F$ with
	$\delta_\ell(F)\ge(\alpha-\eta)\binom{m-\ell}{r}$.
	For $F\in\calH_m$, write
	\[
	\calB(F):=\binom{\binom{V(F)}{\ell}}{b}
	\]
	for the set of all $b$-element families of distinct $\ell$-subsets
	of $V(F)$.
	
	\begin{definition}[The extension family]
		\label{def:extension}
		Let $F\in\calH_m$.  We call $F$ the \emph{base} $k$-graph.  For
		each $B\in\calB(F)$, add a new $r$-set $W_B$, with all these sets
		pairwise disjoint and disjoint from $V(F)$.  Then choose a set
		$P_B\subseteq B$ of size $a+1$, and add the edges $W_B\cup T$ for
		all $T\in P_B$.  All choices of the sets $P_B$ are allowed.  Let
		$\calQ(F)$ be the finite family of all $k$-graphs obtained in this
		way.
	\end{definition}
	Define
	\begin{equation}\label{eq:Falpha-def}
		\calF_\alpha:=\bigcup_{F\in\calH_m}\calQ(F).
	\end{equation}
	
	Note that $\calF_\alpha$ is finite, since $\calH_m$ is finite.  We
	shall prove that $\gamma^{(k)}_\ell(\calF_\alpha)=\alpha$.

	\subsection{The upper bound}
	\label{subsec:rational-upper}
	
	\begin{proposition}
		\label{prop:rational-upper}
		The family $\calF_\alpha$ satisfies
		\[
		\gamma^{(k)}_\ell(\calF_\alpha)\le\alpha.
		\]
	\end{proposition}
	
	\begin{proof}
		It suffices to show that for any $0<\eps<1-\alpha$, every
		sufficiently large $n$-vertex $k$-graph $H$ satisfying $\delta_\ell(H)\ge(\alpha+\eps)\binom{n-\ell}{r}$
		contains a copy of some $k$-graph in $\calF_\alpha$.
		
		By \cref{lem:sampling}, there is an $m$-set $U\subseteq V(H)$ such
		that
		\[
		\delta_\ell(H[U])\ge(\alpha+\eps-\eta)\binom{m-\ell}{r}
		\ge(\alpha-\eta)\binom{m-\ell}{r}.
		\]
		Consequently, $H[U]$ is isomorphic to some $F\in\calH_m$.  We identify this
		copy of $F$ with $H[U]$.
		
		Fix an arbitrary $B=\{T_1,\ldots,T_b\}\in\calB(F)$.  For each $i\in[b]$, all but
		$O(n^{r-1})$ of the $r$-sets counted by $d_H(T_i)$ are disjoint
		from $U$.  Hence, for all sufficiently large $n$, there are at least
		$(\alpha+\eps/2)\binom nr$ sets $R\in\binom{V(H)\setminus U}{r}$
		with $T_i\cup R\in E(H)$.  
		Let $\calR_B$ be the family of all $r$-sets
		$R\subseteq V(H)\setminus U$ that are adjacent to at least $a+1$
		members of $B$. By double counting, we have
		\[
		b|\calR_B|+a\left(\binom{n}{r}-|\calR_B|\right)\ge b\left(\alpha+\frac{\eps}{2}\right)\binom nr,
		\]
		which implies that $|\calR_B|\ge b\eps\binom nr/(2(b-a))=\Omega(n^r)$.
		
		Since the number of members of $\calB(F)$ depends only on $m$, we may choose
		greedily, for every $B\in\calB(F)$, a new $r$-set
		$W_B\subseteq V(H)\setminus U$ such that the chosen $W_B$'s are
		pairwise disjoint and $W_B$ is adjacent to at least $a+1$ members of
		$B$.  Indeed, at each step only $O(1)$ vertices have
		already been used, and only $O(n^{r-1})$ $r$-sets meet them, while there are $\Omega(n^r)$
		available choices.
		
		For each $B$, choose $a+1$ members of $B$ adjacent to $W_B$ and
		call this set $P_B$.  Then the copy $H[U]$, together with the chosen
		sets $W_B$, contains a member of $\calQ(F)\subseteq\calF_\alpha$.
		Therefore every sufficiently large $H$ with $\delta_\ell(H)\ge(\alpha+\eps)\binom{n-\ell}{r}$
		contains a copy of some member of $\calF_\alpha$.  Since $\eps>0$ was arbitrary,
		we have $\gamma^{(k)}_\ell(\calF_\alpha)\le\alpha $.
	\end{proof}

	\subsection{The lower construction}
	\label{subsec:rational-lower-construction}
	
	We now construct  $\calF_\alpha$-free $k$-graphs $G_{\alpha,n}$ on $n$ vertices with
	normalized minimum $\ell$-degree asymptotic to $\alpha$.  The
	construction uses a non-uniform distribution on the color set $\ZZ_Q$,
	where $Q=bM$ has been fixed in \cref{subsec:forbidden-family}.  For all
	sufficiently large $n$, let $G_{\alpha,n}$ be the $k$-graph on
	vertex set $V$ defined as follows.
	
	\begin{itemize}
		\item First define a probability vector $(\mu_0,\ldots,\mu_{Q-1})$ by
		\[
		\mu_i=
		\left(\frac{Q-i}{Q}\right)^{1/r}
		-
		\left(\frac{Q-i-1}{Q}\right)^{1/r},
		\text{ for }0\le i<Q.
		\]
		
		\item Partition $V$ into $Q$ color classes
		$V_0,\ldots,V_{Q-1}$
		such that $|V_i|=\mu_i n+O(1)$ for every $0\le i<Q$.  If
		$v\in V_i$, we write ${\rm col}(v)=i$.
		
		\item Finally define the edges of $G_{\alpha,n}$.  Let
		$K\in\binom{V}k$.  If no color appears in $K$ at least $\ell$
		times, then put $K$ into $G_{\alpha,n}$.  Otherwise, since
		$2\ell>k$, there is a unique color which appears in $K$ at least
		$\ell$ times; call it $c$.  Delete any $\ell$ vertices of color
		$c$ from $K$, denoted by $L$.  Let
		${\rm mcol}(K\setminus L)$ be the minimum color among these remaining $r$ vertices.
		We put  $K$ into $G_{\alpha,n}$ if and only if 	${\rm mcol}(K\setminus L) \notin I_c$.
	\end{itemize}
	
	We remark two key properties of the vector $(\mu_0,\ldots,\mu_{Q-1})$.  Let $X$ be
	a random variable taking values in $\ZZ_Q$ with
	$\mathbb P(X=i)=\mu_i$ for every $0\le i<Q$. Then
	\[
	\mathbb P(X\ge i)=\left(\frac{Q-i}{Q}\right)^{1/r}.
	\]
	Thus, we have $\mu_i>0$ and $\sum_i\mu_i=1$.  
	Since the function
	$x\mapsto x^{1/r}$ is concave on $[0,\infty)$, its largest increment
	on the points $0,1,\ldots,Q$ occurs between $0$ and $1$.  Hence, we have the simple bound
	\begin{equation}
		\max_i\mu_i\le Q^{-1/r}.
		\label{eq:mu-max}
	\end{equation}
	Moreover, if
	$X_1,\ldots,X_r$ are independent copies of $X$, then
	\[
	\mathbb P(\min_j X_j\ge i)
	=
	\prod_{j=1}^r\mathbb P(X_j\ge i)
	=
	\frac{Q-i}{Q},
	\]
	which implies that 
	\begin{equation}
		\mathbb P(\min_j X_j=i)=\frac1Q
		\qquad\text{for every }0\le i<Q.
		\label{eq:uniform-min}
	\end{equation}
	In other words, the minimum of $r$ independent colors is uniform on
	$\ZZ_Q$. This is the reason for choosing the above non-uniform
	distribution.
	
	We shall also use the following consequence of the choice of the color
	classes.  Since $Q$ is fixed and
	$|V_i|=\mu_i n+O(1)$ for every $i\in\ZZ_Q$, sampling any fixed
	number of vertices without replacement has, asymptotically, the same
	color distribution as sampling independently from  $(\mu_0,\ldots,\mu_{Q-1})$.
	More precisely, for every fixed integer $s$, if $Z\subseteq V$ has bounded
	size and $(v_1,\ldots,v_s)$ is a uniformly chosen $s$-tuple of
	distinct vertices from $V\setminus Z$, then
	\[
	\mathbb P\bigl(({\rm col}(v_1),\ldots,{\rm col}(v_s))
	=(i_1,\ldots,i_s)\bigr)
	=
	\prod_{j=1}^s\mu_{i_j}+o(1),
	\]
	where the $o(1)$ term is uniform over all choices of $Z$ with
	$|Z|\le C$ and all color vectors $(i_1,\ldots,i_s)$.

	We next verify the minimum $\ell$-degree of $G_{\alpha,n}$.
	
	\begin{proposition}
		\label{prop:G-degree}
		The $k$-graph $G_{\alpha,n}$ satisfies
		\[
		\delta_\ell(G_{\alpha,n})
		\ge
		(\alpha-o(1))\binom{n-\ell}{r}.
		\]
	\end{proposition}
	
	\begin{proof}
		Fix an arbitrary $\ell$-set $T\subseteq V$. 
		Let $S$ be a uniformly chosen $r$-set from $V\setminus T$.  Since
		$r,\ell$ and $Q$ are fixed, all $o(1)$ terms below are uniform in
		the choice of $T$.
		
		First suppose that $T$ is monochromatic, say $T\subseteq V_c$.  For
		every $r$-set $S\subseteq V\setminus T$, the $k$-set $T\cup S$
		has color $c$ appearing at least $\ell$ times.  By the definition of $G_{\alpha,n}$, $T\cup S\in E(G_{\alpha,n})$ if and only if 
		${\rm mcol}(S)\notin I_c$.
		By \eqref{eq:uniform-min} and the preceding sampling estimate,
		\[
		\mathbb P({\rm mcol}(S) \in I_c)
		=
		\frac{|I_c|}{Q}+o(1)
		=
		\frac hQ+o(1)
		=
		\beta+o(1).
		\]
		Therefore, the number of $r$-sets $S$ for which $T\cup S$ is an edge
		is
		\[
		(1-\beta-o(1))\binom{n-\ell}{r}
		=
		(\alpha-o(1))\binom{n-\ell}{r}.
		\]
		
		Now suppose that $T$ is not monochromatic.  We show that in this case
		$T$ has even larger degree.  If $T\cup S\notin E(G_{\alpha,n})$, then
		some color $i$ appears at least $\ell$ times in $T\cup S$.  Since
		$T$ is not monochromatic and $|S|=r<\ell$, this color must already appear in $T$, and $S$ must contain at
		least one vertex of color $i$.  There are at most $\ell$ possible
		choices for $i$.  By \eqref{eq:mu-max} and a union bound, we get
		\[
		\mathbb P(T\cup S\notin E(G_{\alpha,n}))
		\le
		r\ell\max_i\mu_i+o(1)
		\le
		r\ell Q^{-1/r}+o(1).
		\]
		By the choice of $Q$ in \eqref{eq:Q-large}, this is smaller than
		$\beta/3$ for all sufficiently large $n$.  Consequently,
		\[
		d_{G_{\alpha,n}}(T)
		\ge
		\left(1-\frac{\beta}{3}-o(1)\right)\binom{n-\ell}{r}
		>
		(\alpha-o(1))\binom{n-\ell}{r}.
		\qedhere\]
	\end{proof}

	Finally, we prove the above $k$-graphs $G_{\alpha,n}$ are $\calF_\alpha$-free.

	\begin{proposition}\label{prop:rational-lower}
		For each rational number $\alpha\in (0,1)$, there exists $n_0\in \mathbb N$ such that every $k$-graph $G_{\alpha,n}$ is $\calF_\alpha$-free for $n\ge n_0$.
	\end{proposition}
	
	\begin{proof}
		Suppose, for a contradiction, that $G_{\alpha,n}$ contains a copy of
		some $J\in\calF_\alpha$. Then $J$ is an extension of some
		$F\in\calH_m$.  We identify
		this copy with its image in $G_{\alpha,n}$, and let $U\subseteq V(G_{\alpha,n})$ such that $|U|=m$ and $F\subseteq G_{\alpha,n}[U]$.

		Let $X=\{c\in\ZZ_Q: |U\cap V_c|\ge\ell\}$. Then $X$ is nonempty, otherwise $m\le Q(\ell-1)$, contradicting
		\eqref{eq:m-choice}.
		We next show that $X$ contains a complete residue class modulo $M$.
		By \cref{lem:rho}, it suffices to construct a probability distribution
		$\nu$ supported on $X$ such that $\nu(I_c)<\beta+\rho$ for every $c\in X$.
		
		We first consider a uniform probability measure on $\binom Ur$.  
		For
		$R\in\binom Ur$, write ${\rm mcol}(R):=\min\{{\rm col}(v):v\in R\}$.
		For $A\subseteq\ZZ_Q$, define
		\[
		\nu'(A):=
		\frac{|\{R\in\binom Ur:{\rm mcol}(R)\in A\}|}{\binom mr}.
		\]
		Thus, $\nu'$ is the distribution of the minimum color of a uniformly
		chosen $r$-set from $U$.

		Fix $c\in X$.  By the definition of $X$, we may choose an
		$\ell$-set $T_c\subseteq U\cap V_c$.  
		Consider an $r$-set $R\in\binom{U\setminus T_c}{r}$.  If
		${\rm mcol}(R)\in I_c$, then $T_c\cup R$ is a non-edge of
		$G_{\alpha,n}$.  Since $F\subseteq G_{\alpha,n}[U]$, such an $R$
		is not counted by $d_F(T_c)$.  Therefore,
		\begin{equation}\label{eq:U-minus-T-bound}
			\mathbb P
			\bigl({\rm mcol}(R) \in I_c \mid R\cap T_c=\emptyset \bigr)
			\le
			1-\frac{d_F(T_c)}{\binom{m-\ell}{r}}
			\le
			\beta+\eta.
		\end{equation}
		Moreover, for every fixed
		$v\in T_c$ and every random $r$-set $R\in \binom Ur$, we have $ \mathbb P(v\in R)=\frac rm$
		and hence, by the union bound,
		\[
		\mathbb P(R\cap T_c\ne\emptyset)
		\le
		\sum_{v\in T_c}\mathbb P(v\in R)
		=
		\frac{r\ell}{m}.
		\]
		Combining this with \eqref{eq:U-minus-T-bound}, we obtain
		\begin{equation}\label{eq:nuU-bound}
			\nu'(I_c)=
			\mathbb P
			\left({\rm mcol}(R) \in I_c  \mid R\in \binom Ur \right)
			\le
			\beta+\eta+\frac{r\ell}{m}
			\qquad\text{for every }c\in X.	
		\end{equation}

		The distribution $\nu'$ is not necessarily supported on $X$, so we
		now estimate its mass outside $X$.  Let $Z:=\bigcup_{c\notin X}(U\cap V_c)$.
		Every color outside $X$ appears in $U$ fewer than $\ell$ times,
		and hence $|Z|\le Q(\ell-1)$.
		If ${\rm mcol}(R)\notin X$, then
		$R$ contains a vertex from $Z$.  Consequently,
		\begin{equation}\label{eq:outside-X}
			\nu'(\ZZ_Q\setminus X)
			\le
			\mathbb P(R\cap Z\ne\emptyset)
			\le
			\sum_{v\in Z}\mathbb P(v\in R)
			\le
			\frac{rQ(\ell-1)}{m}
			<1,
		\end{equation}
		where the last inequality follows from \eqref{eq:m-choice}.  Hence
		$\nu'(X)>0$.
		
		Let $\nu$ be the
		distribution $\nu'$ conditioned on the event that the minimum color
		lies in $X$.  That is,
		for $A\subseteq X$, set $\nu(A):=\nu'(A)/\nu'(X)$.
		Then $\nu$ is supported on $X$.  For every $c\in X$, using
		\eqref{eq:nuU-bound}, \eqref{eq:outside-X}, and \eqref{eq:m-choice}, we
		obtain
		\[
		\nu(I_c)
		=
		\nu(I_c\cap X)
		\le
		\frac{\nu'(I_c)}{\nu'(X)}
		\le
		\frac{\beta+\eta+r\ell/m}{1-rQ(\ell-1)/m}
		<
		\beta+\rho.
		\]
		
		By \cref{lem:rho}, the set $X$ contains a complete residue class
		$R_0=\{x\in \mathbb{Z}_Q: x \equiv i \pmod M\}$ for some $i\in  \mathbb{Z}_M$.  Since $Q=bM$,
		we have $ |R_0|=b$.
		For every $c\in R_0$, choose an $\ell$-set $T_c\subseteq U\cap V_c$
		and let $B_0:=\{T_c:c\in R_0\}$.
		Then $B_0\in\calB(F)$.  Since $J$ is an extension of $F$, the
		family $B_0$ has a corresponding new $r$-set $W_{B_0}$, and this
		$r$-set forms edges with $a+1$ members of $B_0$.
		
		Let $S$ be the image of $W_{B_0}$ in $G_{\alpha,n}$, and set $y:={\rm mcol}(S)$.
		By \cref{lem:cyclic}, exactly $b-a$ colors $c\in R_0$ satisfy
		$y\in I_c$.  For each such $c$, the set $T_c$ is monochromatic of
		color $c$, and the definition of $G_{\alpha,n}$ gives
		$T_c\cup S\notin E(G_{\alpha,n})$.
		Therefore, $S$ can form edges with at most
		$ b-(b-a)=a$
		members of $B_0$, contradicting the fact that $W_{B_0}$ forms edges
		with $a+1$ members of $B_0$ in the extension $J$.
		This contradiction shows that $G_{\alpha,n}$ is $\calF_\alpha$-free.
	\end{proof}
	
	\begin{proof}[Proof of \cref{thm:rational-main}]
		If $\alpha=0$, take $\calF$ to consist of one $k$-edge.  Then every
		$\calF$-free $k$-graph is empty, and so
		$\gamma^{(k)}_\ell(\calF)=0$. If $0<\alpha<1$, then construct the finite family
		$\calF_\alpha$ as in \eqref{eq:Falpha-def}.  By
		\cref{prop:rational-upper}, we have $\gamma^{(k)}_\ell(\calF_\alpha)\le \alpha$.
		On the other hand, by \cref{prop:G-degree,prop:rational-lower}, the
		graphs $G_{\alpha,n}$ are $\calF_\alpha$-free for all sufficiently
		large $n$, and satisfy
		\[
		\delta_\ell(G_{\alpha,n})
		\ge
		(\alpha-o(1))\binom{n-\ell}{r}.
		\]
		Thus, $\gamma^{(k)}_\ell(\calF_\alpha)\ge\alpha $.
		Hence $\gamma^{(k)}_\ell(\calF_\alpha)=\alpha$, and 
		$\alpha\in\Gamma^{k,\ell}_{\rm fin}$.
	\end{proof}

\section{A color-forcing construction}\label{sec:qcolor}
	
In this section, we prove \cref{thm:q-color-main} by a color-forcing construction.  
	Given an integer $q\ge 1$, we say that a $k$-graph $F$ is \emph{$q$-colorable} if $V(F)$ can be
	partitioned into $q$ sets $V_1,\ldots,V_q$ such that no edge of $F$ lies entirely inside one of the parts. Then we have the following observation.
	
	\begin{observation}	\label{obs:coloring-lower}
Let $k>\ell\ge1$ and $q\ge1$.  If a $k$-graph $F$ is not $q$-colorable, then
		\[
		\gamma^{(k)}_\ell(F)\ge 1-q^{\ell-k}.
		\]
	\end{observation}
	
	\begin{proof}
Let $H$ be the $n$-vertex $k$-graph obtained from a balanced partition $V(H)=V_1\cup\cdots\cup V_q$ by taking as edges all $k$-sets not contained in a single part.
For any $\ell$-set $T\subseteq V(H)$, if $T$ is contained in some part $V_i$, then 
		\[
		d_H(T)\ge
		\binom{n-\ell}{k-\ell}
		-
		\binom{\lceil n/q\rceil-\ell}{k-\ell}
		=
		\left(1-q^{\ell-k}+o(1)\right)\binom{n-\ell}{k-\ell}.
		\]
		If $T$ is not contained in one part, then every $(k-\ell)$-set extends $T$
		to an edge of $H$.  Letting $n\to\infty$ gives the result.
	\end{proof}
	
	We now construct a $k$-graph $F_{\ell,q}^{(k)}$  for any $q\ge 1$ such that $\gamma^{(k)}_\ell(F_{\ell,q}^{(k)})=1-q^{\ell-k}$.

\begin{definition}	\label{def:Fq}
Fix integers $q\ge1$ and $k>\ell\ge1$. The $k$-graph $F_{\ell,q}^{(k)}$ is defined recursively as follows. When $\ell=1$, the recursion involves $1$-graphs at its
initial stage; as usual, a $1$-edge is a singleton.	
		\begin{enumerate}[label=\textup{(\roman*)}]
			\item For $q=1$ and $\ell\le j\le k$,  let $F^{(j)}_{\ell,1}$ be the $j$-graph consisting of one $j$-edge;
			
			\item For $q\ge2$, let $F^{(\ell)}_{\ell, q}$ be the complete $\ell$-graph $K^{(\ell)}_{q(\ell-1)+1}$ on $q(\ell-1)+1$ vertices;
			
			\item For $q\ge2$ and $\ell< j\le k$, define $F^{(j)}_{\ell,q}$ recursively as follows. Assume we have defined the $(j-1)$-graph $F^{(j-1)}_{\ell,q}$.
			Let  $m:=|E(F^{(j-1)}_{\ell,q})|$ and write
			$E(F^{(j-1)}_{\ell, q})=\{ e^{(j-1)}_1, e^{(j-1)}_2,\dots, e^{(j-1)}_{m}\}$. For each $i\in [m]$, take a new copy $S_i$ of $F^{(j)}_{\ell,q-1}$,  all these copies being pairwise disjoint and disjoint from the set $V(F^{(j-1)}_{\ell, q})$. Define $F^{(j)}_{\ell,q}$ by
			\[
			V(F^{(j)}_{\ell,q})
			:=
			V(F^{(j-1)}_{\ell,q})\cup \bigcup_{i=1}^{m} V(S_i),
			\]
			and 
			\[
			E(F^{(j)}_{\ell,q})
			:=
			\bigcup_{i=1}^{m} E(S_i)
			\cup
			\{e^{(j-1)}_i \cup\{v\}:  i\in [m],  v\in V(S_i)\}.
			\]
		\end{enumerate}
	\end{definition}

	\begin{lemma}\label{lem:C-not-colorable}
		For every $q\ge1$ and every $k>\ell\ge 1$, the $k$-graph $F^{(k)}_{\ell,q}$ is not $q$-colorable.
	\end{lemma}
	
	\begin{proof}
		The case $q=1$ is immediate, since $F^{(k)}_{\ell, 1}$ consists of one edge.
		Let $q\ge2$.  The initial $\ell$-graph
		$F^{(\ell)}_{\ell, q}=K^{(\ell)}_{q(\ell-1)+1}$ is not $q$-colorable by the pigeonhole principle.
		
		Next consider $\ell< j\le k$.  Suppose, for a contradiction, that
		$F^{(j)}_{\ell, q}$  has a $q$-coloring.  By induction on $j$, the base copy $F^{(j-1)}_{\ell, q}$ contains a monochromatic
		$(j-1)$-edge $e^{(j-1)}_{i_0}$, say of color $\alpha$.  Consider the attached copy $S_{i_0}$ of
		$F^{(j)}_{\ell, q-1}$.  If some vertex $v\in V(S_{i_0})$ has color $\alpha$, then
		$e^{(j-1)}_{i_0}\cup\{v\}$ is a monochromatic $j$-edge.  Otherwise all vertices of
		$S_{i_0}$ receive colors different from $\alpha$, so $S_{i_0}$ is colored with
		at most $q-1$ colors. 
		By induction on $q$, $S_{i_0}$ contains a
		monochromatic $j$-edge. 
		In both cases we get a contradiction.
	\end{proof}
	
\begin{proposition}\label{prop:C-upper}
		For every $q\ge1$, $k>\ell\ge 1$ and every $\varepsilon>0$, every sufficiently large
		$n$-vertex $k$-graph $H$ with
		$\delta_\ell(H)\ge
		\left(1-q^{\ell-k}+\varepsilon\right)\binom{n-\ell}{k-\ell}$
		contains a copy of $F^{(k)}_{\ell,q}$.
	\end{proposition}

	\begin{proof}
		Let $r:=k-\ell$.  We prove the proposition
		by induction on $q+r$.  
 If $q=1$, then $F^{(k)}_{\ell,1}$ is a single edge, and the conclusion is immediate for large $n$.  Hence, assume $q\ge2$.		

First consider the case $r=1$, that is, $k=\ell+1$.  Let $t=|V(F^{(k)}_{\ell,q})|$.
		We claim that for every $\ell$-set $Y\subseteq V(H)$ and every
		$W\subseteq V(H)$ with $|W|\le t$, the induced $k$-graph
		$H[N_H(Y)\setminus W]$
		contains a copy of $F^{(k)}_{\ell,q-1}$.
	Let $U:=N_H(Y)\setminus W$.  Since $r=1$, the minimum degree
		assumption gives, for all sufficiently large $n$, $|U|\ge
		\left(1-\frac1q+\frac{\varepsilon}{2}\right)n$.
		For every $\ell$-set $T\subseteq U$, 
		\[
		d_{H[U]}(T)\ge d_H(T)-(n-|U|)
		\ge |U|-\frac nq+\frac{\varepsilon n}{2}
		\]
		for all sufficiently large $n$.  Since
	$|U|\ge
	\left(1-\frac1q+\frac{\varepsilon}{2}\right)n$, 
		there exists $\varepsilon'>0$, depending only on $q,\ell$ and
		$\varepsilon$, such that
		\[
		d_{H[U]}(T)
		\ge
		\left(1-\frac1{q-1}+\varepsilon'\right)(|U|-\ell).
		\]
 By the induction hypothesis
		applied with $q-1$, the $k$-graph $H[U]$ contains a copy of
		$F^{(k)}_{\ell,q-1}$.  This proves the claim.
		
		Now choose a set $S\subseteq V(H)$ of size $q(\ell-1)+1$.  This set
		plays the role of the initial complete $\ell$-graph
		$F^{(\ell)}_{\ell,q}=K^{(\ell)}_{q(\ell-1)+1}$.  Enumerate the
		$\ell$-subsets of $S$ as $Y_1,\ldots,Y_m$.  We process these
		$\ell$-sets one by one.  When $Y_i$ is processed, let $W$ be the
		set of vertices already used.  Since $|W|\le t$, the claim gives a
		copy $S_i$ of $F^{(k)}_{\ell,q-1}$ inside
		$H[N_H(Y_i)\setminus W]$.  Thus the copies $S_i$ can be chosen
		pairwise disjoint and disjoint from $S$, and for every
		$v\in V(S_i)$ we have
$Y_i\cup\{v\}\in E(H)$.
		This is exactly the recursive construction of $F^{(k)}_{\ell,q}$ in
		the case $k=\ell+1$.  Hence $H$ contains $F^{(k)}_{\ell,q}$.
		
Next, we consider the case $r\ge2$.  Again let
		$t:=|V(F^{(k)}_{\ell,q})|$.  We define an auxiliary $(k-1)$-graph
		$G$ on $V(H)$ as follows.  A $(k-1)$-set $Y\in E(G)$ if for
		every set $W\subseteq V(H)$ with $|W|\le t$, the induced $k$-graph $H[N_H(Y)\setminus W]$
		contains a copy of $F^{(k)}_{\ell,q-1}$.
		
		We first observe that if $G$ contains a copy of
		$F^{(k-1)}_{\ell,q}$, then $H$ contains a copy of
		$F^{(k)}_{\ell,q}$.  Indeed, let $S$ be a copy of
		$F^{(k-1)}_{\ell,q}$ in $G$, and enumerate its edges as
		$Y_1,\ldots,Y_m$.  We process these edges in this order.  When $Y_i$
		is processed, let $W$ be the set of vertices already used.  Since
		$Y_i\in E(G)$, we can choose a copy $S_i$ of
		$F^{(k)}_{\ell,q-1}$ inside $H[N_H(Y_i)\setminus W]$.  The copies
		$S_i$ are pairwise disjoint and disjoint from $S$, and every
		$v\in V(S_i)$ satisfies $Y_i\cup\{v\}\in E(H)$.
		Thus, the recursive construction of $F^{(k)}_{\ell,q}$ is shown in
		$H$.  Therefore, it suffices to show that $G$ contains $F^{(k-1)}_{\ell,q}$.  
		
Choose constants $0<\eta,\theta\ll\varepsilon$.  We first prove that
		every non-edge $Y$ of $G$ satisfies
		\begin{equation}
			d_H(Y)\le \left(\frac{q-1}{q}+\theta\right)n
			\label{eq:nonedge-Y-small-neighbourhood}
		\end{equation}
		for all sufficiently large $n$.  Suppose not and let $Y$ be a
		non-edge of $G$ with $d_H(Y)>((q-1)/q+\theta)n$.
		Let $W\subseteq V(H)$ be any set with $|W|\le t$, and set
		$U:=N_H(Y)\setminus W$.  Then, for all sufficiently large $n$,
		$|U|\ge((q-1)/q+\theta/2)n$.
For every $\ell$-set $T\subseteq U$, 
		\[
		d_{H[U]}(T)
		\ge
		\delta_\ell(H)-
		\left(\binom{n-\ell}{r}-\binom{|U|-\ell}{r}\right).
		\]
		Using the minimum degree assumption on $H$ and the lower bound on
		$|U|$, we obtain
		\[
		d_{H[U]}(T)
		\ge
		\left(1-(q-1)^{-r}+\varepsilon''\right)
		\binom{|U|-\ell}{r}
		\]
		for some $\varepsilon''>0$ depending only on
		$q,r,\theta$ and $\varepsilon$.  By the induction hypothesis applied
		with $q-1$, $H[U]$ contains a copy of
		$F^{(k)}_{\ell,q-1}$.
	Since this argument holds for every $W\subseteq V(H)$ with
		$|W|\le t$, the set $Y$ is an edge of $G$, contradicting our choice
		of $Y$.  This proves \eqref{eq:nonedge-Y-small-neighbourhood}.
		
We next prove that $G$ has large minimum $\ell$-degree. Suppose, for a contradiction, that some $\ell$-set $T\subseteq V(H)$ satisfies
$	d_G(T)\le
\left(1-q^{-(r-1)}+\eta\right)\binom{n-\ell}{r-1}$.	Double-counting pairs $(Z,v)$ with $Z\in\binom{V(H)\setminus T}{r-1}$ and $T\cup Z\cup\{v\}\in E(H)$ gives
\[
\sum_{Z\in\binom{V(H)\setminus T}{r-1}} d_H(T\cup Z)=r \cdot d_H(T).
\]
Hence, by the minimum $\ell$-degree assumption on $H$, the average of $d_H(T\cup Z)$ over all such $Z$ is at least $(1-q^{-r}+\eps/2)n$ for all large $n$.
On the other hand, if $T\cup Z\in E(G)$, then trivially
		$d_H(T\cup Z)\le n$.  If $T\cup Z\notin E(G)$, then
		\eqref{eq:nonedge-Y-small-neighbourhood} gives $d_H(T\cup Z)\le ((q-1)/q+\theta )n$.
		Using the assumed upper bound on $d_G(T)$, the average of $d_H(T\cup Z)$ over all such $Z$ is at most
		\[
		\left(1-q^{-(r-1)}+\eta\right)n
		+
		\left(q^{-(r-1)}-\eta\right)
		\left(\frac{q-1}{q}+\theta\right)n.
		\]
		If $\eta$ and $\theta$ are chosen sufficiently small in terms of
		$\varepsilon$, this is at most $(1-q^{-r}+{\varepsilon}/{4})n$,
a contradiction.  Thus, 
		\[
		\delta_\ell(G)\ge
		\left(1-q^{-(r-1)}+\eta\right)\binom{n-\ell}{r-1}.
		\]
		By the induction hypothesis, $G$ contains a copy of
		$F^{(k-1)}_{\ell,q}$.  As observed above, this copy lifts to a copy of
		$F^{(k)}_{\ell,q}$ in $H$.
	\end{proof}
	
\begin{proof}[Proof of~\cref{thm:q-color-main}]
Let $F^{(k)}_{\ell,q}$ be the $k$-graph defined in \cref{def:Fq}.
By \cref{lem:C-not-colorable}, it is not $q$-colorable, and hence
\cref{obs:coloring-lower} gives
\[
  \gamma^{(k)}_\ell(F^{(k)}_{\ell,q})\ge 1-q^{\ell-k}.
\]
The reverse inequality follows from \cref{prop:C-upper}.
\end{proof}

\section{Stability for the two-color forcing hypergraph}\label{sec:stability}

The exact value in \cref{thm:q-color-main} for $q=2$ is attained by the balanced bipartite construction: take all $k$-sets that meet both parts.  The next lemma extracts the structural feature of this construction that will be needed in the proof of non-principality.  It does not assert full edit-distance stability; rather, it guarantees two disjoint independent sets that together cover almost all vertices.

\begin{lemma}[Stability]\label{lem:stability}
Let $k>\ell>1$.  For every $\xi>0$ there exists $\eps>0$ such that the following holds for all sufficiently large $n$.  If $H$ is an $n$-vertex $F^{(k)}_{\ell,2}$-free $k$-graph satisfying
\begin{equation}\label{eq:stability-assumption}
  \delta_\ell(H)\ge \left(1-2^{\ell-k}-\eps\right)\binom{n-\ell}{k-\ell},
\end{equation}
then $H$ contains disjoint independent sets $A$ and $B$ with
$|A|,|B|\ge(1/2-\xi)n$.
\end{lemma}

\begin{proof}
Choose constants
$0<\eps\ll\eta\ll\theta\ll\zeta\ll\xi$, and let
$L=|V(F^{(k)}_{\ell,2})|$. Set $r=k-\ell$. We first handle the case $r=1$ and then assume $r\ge2$.

Suppose first that $r=1$.  Fix a set $X\subseteq V(H)$ of size $2\ell-1$.  We claim that there are an $\ell$-set $Y\subseteq X$ and a set $Q\subseteq V(H)$ with $|Q|\le L$ such that $H[N_H(Y)\setminus Q]$ contains no edge.  Otherwise, enumerate the $\ell$-subsets of $X$ and process them one at a time.  For each such set $Y$, take $Q$ to consist of $X$ and all vertices used at earlier steps, and choose an edge of $H[N_H(Y)\setminus Q]$.  The chosen edges are pairwise disjoint and, together with the complete $\ell$-graph on $X$, form a copy of $F^{(\ell+1)}_{\ell,2}$, a contradiction.

Choose $Y$ and $Q$ as in the claim and set $A=N_H(Y)\setminus Q$.  Then $A$ is independent and, by \eqref{eq:stability-assumption},
$|A|\ge(1/2-\xi/2)n$ for all sufficiently large $n$.  Choose a $(2\ell-1)$-set $X'\subseteq A$ and apply the same claim to $X'$.  This gives an $\ell$-set $Y'\subseteq X'$ and a set $Q'$ with $|Q'|\le L$ such that $B=N_H(Y')\setminus Q'$ is independent and $|B|\ge(1/2-\xi/2)n$.  Since $Y'\subseteq A$ and $A$ is independent, $N_H(Y')\cap A=\emptyset$; hence $A\cap B=\emptyset$.  This proves the case $r=1$.

Now let $r\ge2$.  Define an auxiliary $(k-1)$-graph $G$ on $V(H)$ by declaring a $(k-1)$-set $Y$ to be an edge if, for every $Q\subseteq V(H)$ with $|Q|\le L$, the induced $k$-graph $H[N_H(Y)\setminus Q]$ contains an edge.  If $G$ contained $F^{(k-1)}_{\ell,2}$, then one could process the edges of this copy and greedily attach a fresh edge inside the corresponding set $N_H(Y)$.  This is exactly the recursion in \cref{def:Fq} for $F^{(k)}_{\ell,2}$.  Thus $G$ is $F^{(k-1)}_{\ell,2}$-free.

By \cref{thm:q-color-main},
$\gamma^{(k-1)}_\ell(F^{(k-1)}_{\ell,2})=1-2^{-(r-1)}$.  Therefore, for all sufficiently large $n$, some $\ell$-set $T\subseteq V(H)$ satisfies
\begin{equation}\label{eq:low-G-star}
  d_G(T)\le\left(1-2^{-(r-1)}+\eta\right)
  \binom{n-\ell}{r-1}.
\end{equation}
Double-counting the pairs $(Z,v)$ for which
$Z\in\binom{V(H)\setminus T}{r-1}$ and
$T\cup Z\cup\{v\}\in E(H)$ gives
\begin{equation}\label{eq:first-average}
 \frac{1}{\binom{n-\ell}{r-1}}
 \sum_{Z\in\binom{V(H)\setminus T}{r-1}}d_H(T\cup Z)
 \ge \left(1-2^{-r}-2\eps\right)n.
\end{equation}
Here we used
$r\binom{n-\ell}{r}/\binom{n-\ell}{r-1}=n-k+1$.

We claim that some non-edge $Y=T\cup Z$ of $G$ satisfies
$d_H(Y)\ge(1/2-\theta)n$.  Indeed, otherwise \eqref{eq:low-G-star} would bound the average in \eqref{eq:first-average} from above by
\[
 \left(1-2^{-(r-1)}+\eta\right)n
 +\left(2^{-(r-1)}-\eta\right)\left(\frac12-\theta\right)n,
\]
which equals
$\bigl(1-2^{-r}+\eta/2-\theta(2^{-(r-1)}-\eta)\bigr)n$ and is smaller than the right-hand side of \eqref{eq:first-average}.  This proves the claim.

Because $Y\notin E(G)$, there is a set $Q$ with $|Q|\le L$ such that
$A=N_H(Y)\setminus Q$ is independent.  The claim gives
$|A|\ge(1/2-2\theta)n$ for large $n$.  On the other hand, if
$U\in\binom A\ell$, then every $r$-set in $A\setminus U$ is a non-neighbor of $U$.  Combining this with \eqref{eq:stability-assumption} yields
$\binom{|A|-\ell}{r}\le(2^{-r}+\eps)\binom{n-\ell}{r}$.
Consequently,
\begin{equation}\label{eq:A-size}
  \left(\frac12-2\theta\right)n\le |A|\le
  \left(\frac12+\theta\right)n
\end{equation}
for all sufficiently large $n$.

The induced auxiliary graph $G[A]$ is also $F^{(k-1)}_{\ell,2}$-free.  Writing $m=|A|$ and applying \cref{thm:q-color-main} inside $A$, we find an $\ell$-set $T'\subseteq A$ such that
\begin{equation}\label{eq:low-GA-star}
 d_{G[A]}(T')\le\left(1-2^{-(r-1)}+\eta\right)
 \binom{m-\ell}{r-1}.
\end{equation}
We next estimate the average of $d_H(T'\cup Z)$ over
$Z\in\binom{A\setminus T'}{r-1}$.  Since $A$ is independent, all
$\binom{m-\ell}{r}$ extensions of $T'$ that lie entirely in $A$ are missing.  Hence the number of missing extensions that are not entirely in $A$ is
\[
 D:=\binom{n-\ell}{r}-\binom{m-\ell}{r}-d_H(T')
 \le(2^{-r}+\eps)\binom{n-\ell}{r}-\binom{m-\ell}{r}.
\]
By \eqref{eq:A-size}, $m=(1/2+O(\theta))n$, so
$D=O((\theta+\eps)n^r)$.  The sum of $d_H(T'\cup Z)$ over
$Z\in\binom{A\setminus T'}{r-1}$ counts the edges containing $T'$ whose remaining vertices consist of $r-1$ vertices in $A$ and one vertex outside $A$.  There are
$(n-m)\binom{m-\ell}{r-1}$ possible extensions of this type, and at most $D$ of them are missing.  Therefore
\begin{equation}\label{eq:second-average}
 \frac{1}{\binom{m-\ell}{r-1}}
 \sum_{Z\in\binom{A\setminus T'}{r-1}}d_H(T'\cup Z)
 \ge n-m-O((\theta+\eps)n).
\end{equation}

Every $(k-1)$-set contained in $A$ has at most $n-m$ neighbors in $H$.  If every non-edge $Y'=T'\cup Z$ of $G[A]$ had
$d_H(Y')<n-m-\zeta n$, then \eqref{eq:low-GA-star} would make the average in \eqref{eq:second-average} at most
$n-m-\zeta(2^{-(r-1)}-\eta)n$, a contradiction.  Thus there is a non-edge $Y'$ of $G[A]$ with
$d_H(Y')\ge n-m-\zeta n$.

Since $G[A]$ is induced, $Y'$ is also a non-edge of $G$.  Hence some set $Q'$ with $|Q'|\le L$ makes
$B=N_H(Y')\setminus Q'$ independent.  Moreover, $B\cap A=\emptyset$, because $Y'\subseteq A$ and $A$ is independent.  Finally, by \eqref{eq:A-size},
$|B|\ge n-m-\zeta n-L\ge(1/2-\xi)n$ for large $n$, while
$|A|\ge(1/2-\xi)n$.  This completes the proof.
\end{proof}

\section{A non-principal pair}\label{sec:nonprincipal}

We now combine the stability lemma with a tournament construction of Mubayi and Zhao~\cite{MubayiZhao2007}.  For $t\ge k$, let $K^k(t,t)$ be the $k$-graph with two disjoint vertex classes of size $t$ whose edges are precisely the $k$-sets meeting one of the two classes in exactly one vertex.

\begin{lemma}\label{lem:K-lower}
Let $k\ge3$ and $1<\ell<k$, and put $r=k-\ell$.  There exists $t=t(k)$ such that
\[
  \gamma^{(k)}_\ell(K^k(t,t))\ge1-2^{-r}.
\]
\end{lemma}

\begin{proof}
Let $R$ be a random tournament on $[n]$.  Define a $3$-graph $G_3$ as follows: for $i<j<m$, declare $\{i,j,m\}$ to be an edge when exactly one of the tournament edges $i\to j$ and $i\to m$ is present.  For a pair $P$, let
$N_3(P)=\{x\notin P:P\cup\{x\}\in E(G_3)\}$.  Conditional on the orientation of the tournament edge spanned by $P$, the indicators of the events $x\in N_3(P)$ are independent Bernoulli random variables with parameter $1/2$.  Chernoff's inequality and a union bound therefore show that, for every fixed $\xi>0$ and all sufficiently large $n$, one can choose $R$ so that
\begin{equation}\label{eq:pairs-good}
  |N_3(P)|\ge\left(\frac12-\xi\right)n
  \quad\text{for every }P\in\binom{[n]}2.
\end{equation}
Fix such a tournament, and let $H_n$ be the $k$-graph whose edges are the $k$-sets containing an edge of $G_3$.

The $3$-graph $G_3$ is $K^{(3)}_4$-free.  Indeed, if
$a<b<c<d$ and the triples $abc$, $abd$, and $acd$ were all edges, then the three orientations from $a$ to $b,c,d$ would have to be pairwise different, which is impossible.  Let $R_3(4,k-1)$ denote the least integer such that every $3$-graph on that many vertices contains either a $K^{(3)}_4$ or an independent set of size $k-1$, and set
$t=2R_3(4,k-1)$.

We claim that $H_n$ is $K^k(t,t)$-free.  Suppose that $A$ and $B$ are disjoint $t$-sets, and let $a_0$ be the smallest vertex of $A\cup B$; by symmetry, assume that $a_0\in A$.  At least half of the vertices of $B$ have the same orientation toward $a_0$.  Among those vertices, Ramsey's theorem gives a $(k-1)$-set $B_0$ spanning no edge of $G_3$.  Every triple consisting of $a_0$ and two vertices of $B_0$ is also absent from $G_3$, since the two tournament edges from $a_0$ have the same orientation.  Thus $\{a_0\}\cup B_0$ contains no edge of $G_3$ and is a non-edge of $H_n$.  It is one of the edges required in a copy of $K^k(t,t)$ with classes $A$ and $B$, proving the claim.

It remains to estimate the minimum $\ell$-degree of $H_n$.  Let
$T\in\binom{[n]}\ell$.  If $T$ already contains an edge of $G_3$, then every $r$-set extends $T$ to an edge of $H_n$.  Otherwise choose a pair $P\subseteq T$.  Every $r$-set meeting $N_3(P)$ extends $T$ to an edge, and \eqref{eq:pairs-good} gives
\[
 d_{H_n}(T)\ge
 \left(1-\left(\frac12+\xi+o(1)\right)^r\right)
 \binom{n-\ell}{r}.
\]
Since $\xi>0$ is arbitrary, the desired lower bound follows.
\end{proof}

\begin{lemma}\label{lem:joint-positive}
Let $k\ge3$ and $1<\ell<k$.  If $t=t(k)$ is chosen as in \cref{lem:K-lower}, then
\[
 \gamma^{(k)}_\ell\bigl(\{K^k(t,t),F^{(k)}_{\ell,2}\}\bigr)>0.
\]
\end{lemma}

\begin{proof}
Fix a balanced partition $[n]=A\cup B$ and expose a random tournament as in the proof of \cref{lem:K-lower}.  The same concentration argument, now applied separately inside $A$ and $B$, gives a tournament for which
\begin{equation}\label{eq:split-neighborhood}
 |N_3(P)\cap A|\ge n/5
 \quad\text{and}\quad
 |N_3(P)\cap B|\ge n/5
 \quad\text{for every }P\in\binom{[n]}2.
\end{equation}
Let $H_n$ again consist of the $k$-sets containing an edge of $G_3$, and let $H_n^*$ be the subgraph formed by the edges that meet both $A$ and $B$.

The $k$-graph $H_n^*$ is $2$-colorable, so it is $F^{(k)}_{\ell,2}$-free by \cref{lem:C-not-colorable}.  It is also $K^k(t,t)$-free because it is a subgraph of $H_n$.  To estimate its minimum $\ell$-degree, fix $T\in\binom{[n]}\ell$ and a pair $P\subseteq T$.  If $T$ meets both parts, then every $r$-set meeting $N_3(P)$ gives an edge of $H_n^*$, and \eqref{eq:split-neighborhood} yields
$d_{H_n^*}(T)\ge(1-(3/5+o(1))^r)\binom{n-\ell}{r}$.
If $T\subseteq A$, it is enough that the extending set meet $N_3(P)\cap B$, giving
$d_{H_n^*}(T)\ge(1-(4/5+o(1))^r)\binom{n-\ell}{r}$; the case $T\subseteq B$ is symmetric.  Hence, for all sufficiently large $n$,
\[
 \delta_\ell(H_n^*)\ge
 \frac12\left(1-\left(\frac45\right)^r\right)
 \binom{n-\ell}{r}>0,
\]
which proves the lemma.
\end{proof}

\begin{proof}[Proof of \cref{thm:nonprincipal-main}]
Set $r=k-\ell$, let $t=t(k)$ be given by \cref{lem:K-lower}, and set
$F_1=K^k(t,t)$ and $F_2=F^{(k)}_{\ell,2}$.  By \cref{lem:joint-positive},
$\gamma^{(k)}_\ell(\{F_1,F_2\})>0$.  Moreover, \cref{lem:K-lower,thm:q-color-main} give
\[
 \gamma^{(k)}_\ell(F_1)\ge1-2^{-r}
 \quad\text{and}\quad
 \gamma^{(k)}_\ell(F_2)=1-2^{-r}.
\]
It remains to prove that the density of the pair is strictly smaller than $1-2^{-r}$.

Choose $\xi>0$ sufficiently small in terms of $k,\ell$, and $t$.  Let $\eps_0>0$ be supplied by \cref{lem:stability}, and then choose
$0<\eps\le\eps_0$ sufficiently small.  We claim that every sufficiently large $n$-vertex $k$-graph $H$ satisfying
\begin{equation}\label{eq:final-degree}
 \delta_\ell(H)\ge\left(1-2^{-r}-\eps\right)
 \binom{n-\ell}{r}
\end{equation}
contains $F_1$ or $F_2$.

Assume that $H$ is $F_2$-free.  By \cref{lem:stability}, there are disjoint independent sets $A,B\subseteq V(H)$ with
$(1/2-\xi)n\le |A|,|B|\le(1/2+\xi)n$.  We show that $H$ contains $K^k(t,t)$ with one class in $A$ and the other in $B$.

Let $\calM_A$ be the family of missing $k$-sets having $k-1$ vertices in $A$ and one in $B$.  For $T\in\binom A\ell$, let $X_T$ be the number of non-neighbors of $T$ that are not contained entirely in $A$.  Since $A$ is independent, \eqref{eq:final-degree} gives
\begin{equation*}
 X_T\le(2^{-r}+\eps)\binom{n-\ell}{r}
      -\binom{|A|-\ell}{r}
 \le C_0(\xi+\eps)n^r
\end{equation*}
for a constant $C_0=C_0(k,\ell)$.  Counting pairs $(T,e)$ with
$T\subseteq e\cap A$, where $T\in\binom A\ell$ and $e\in\calM_A$, yields
$|\calM_A|\le C_1(\xi+\eps)n^k$ for a constant $C_1=C_1(k,\ell)$.  By symmetry, the same bound holds for the family $\calM_B$ of missing $k$-sets having $k-1$ vertices in $B$ and one in $A$.

Choose $A_0\in\binom At$ and $B_0\in\binom Bt$ independently and uniformly.  For each $e\in\calM_A\cup\calM_B$, the probability that $e\subseteq A_0\cup B_0$ is at most $C_2n^{-k}$, where $C_2=C_2(k,t)$.  Hence the expected number of missing edges required for a copy of $K^k(t,t)$ on $A_0\cup B_0$ is at most
$2C_1C_2(\xi+\eps)$.  Choosing $\xi$ and then $\eps$ so that this quantity is less than one, we obtain $A_0$ and $B_0$ for which no required edge is missing.  Thus $H[A_0\cup B_0]$ contains $K^k(t,t)$, proving the claim.

The preceding claim implies that $\gamma^{(k)}_\ell(\{F_1,F_2\})
 \le1-2^{-r}-\eps<1-2^{-r}$.
Together with the positive lower bound and the estimates for the individual forbidden graphs, this proves
\[
 0<\gamma^{(k)}_\ell(\{F_1,F_2\})
 <\min\{\gamma^{(k)}_\ell(F_1),\gamma^{(k)}_\ell(F_2)\}.
\qedhere\]
\end{proof}

\section{Concluding remarks}\label{sec:concluding-remarks}

The condition $\ell>k/2$ in \cref{thm:rational-main} is used only in the lower-bound construction.  It guarantees that a $k$-set has at most one color of multiplicity at least $\ell$, and hence a single cyclic interval determines whether the set is an edge.  Removing this uniqueness condition appears to be the main obstacle to extending the rational realization theorem to the full range.

\begin{conjecture}\label{conj:rational-full-range}
Let $k\ge3$ and $1<\ell<k$.  Every rational number $\alpha\in[0,1)$ is equal to $\gamma^{(k)}_\ell(\calF)$ for some finite family $\calF$ of $k$-graphs.
\end{conjecture}

A sharper problem is to determine the set of values realized by a single
forbidden $k$-graph.  The color-forcing construction gives the infinite
sequence $1-q^{\ell-k}$ for every $k>\ell\ge1$, whereas
\cref{thm:rational-main} realizes all rational values by finite families
when $\ell>k/2$.  For $\ell=1$, the parameter
$\gamma^{(k)}_1$ is the classical Tur\'an density $\pi$, so the corresponding
single-forbidden problem belongs to the classical hypergraph Tur\'an density
spectrum.  It would be particularly interesting to decide whether, for any
fixed pair $(k,\ell)$ with $k>\ell>1$, every rational number in $[0,1)$ can
be realized by a single forbidden $k$-graph.

\section*{Acknowledgments and AI disclosure}
Quanyu Tang would like to thank Professor Guanghui Wang for his warm hospitality during Tang's visit to Shandong University. Wenling Zhou was supported by the National Natural Science Foundation of China (12401457), the China Postdoctoral Science Foundation (2024M761780), the Natural Science Foundation of Shandong Province (ZR2024QA067), and the Young Talent of Lifting Engineering for Science and Technology in Shandong, China (SDAST2025QTA074).

The authors used AI for language editing and exploratory discussion of proof presentation.  All mathematical statements, proofs, and references were independently written and finalized by the authors, who take full responsibility for the content of the paper.

\bibliographystyle{abbrv}
\bibliography{ref}

@article{Balogh2002,
  author  = {Balogh, J{\'o}zsef},
  title   = {The {Tur{\'a}n} density of triple systems is not principal},
  journal = {J. Combin. Theory Ser. A},
  volume  = {100},
  number  = {1},
  pages   = {176--180},
  year    = {2002}
}

@incollection{BaloghClemenLidicky2022,
  author    = {Balogh, J{\'o}zsef and Clemen, Felix Christian and Lidick{\'y}, Bernard},
  title     = {Hypergraph {Tur{\'a}n} problems in {$\ell_2$}-norm},
  booktitle = {Surveys in Combinatorics 2022},
  series    = {London Math. Soc. Lecture Note Ser.},
  volume    = {481},
  pages     = {21--63},
  publisher = {Cambridge Univ. Press},
  address   = {Cambridge},
  year      = {2022}
}

@article{ErdosSimonovits1966,
  author  = {Erd{\H{o}}s, P. and Simonovits, M.},
  title   = {A limit theorem in graph theory},
  journal = {Studia Sci. Math. Hungar.},
  volume  = {1},
  pages   = {51--57},
  year    = {1966}
}

@article{ErdosSos1982,
  author  = {Erd{\H{o}}s, P. and S{\'o}s, Vera T.},
  title   = {On {Ramsey--Tur{\'a}n} type theorems for hypergraphs},
  journal = {Combinatorica},
  volume  = {2},
  number  = {3},
  pages   = {289--295},
  year    = {1982}
}

@article{ErdosStone1946,
  author  = {Erd{\H{o}}s, P. and Stone, A. H.},
  title   = {On the structure of linear graphs},
  journal = {Bull. Amer. Math. Soc.},
  volume  = {52},
  pages   = {1087--1091},
  year    = {1946}
}

@article{FranklRodl1984,
  author  = {Frankl, Peter and R{\"o}dl, Vojt{\v{e}}ch},
  title   = {Hypergraphs do not jump},
  journal = {Combinatorica},
  volume  = {4},
  number  = {2--3},
  pages   = {149--159},
  year    = {1984}
}

@article{Langliling23,
  author  = {Lang, Richard},
  title   = {Tiling dense hypergraphs},
  journal = {arXiv preprint},
  note    = {arXiv:2308.12281},
  year    = {2023}
}

@article{GaoPikhurkoRongSun2026,
  author  = {Gao, Jun and Pikhurko, Oleg and Rong, Mingyuan and Sun, Shumin},
  title   = {Rational codegree {Tur{\'a}n} density of hypergraphs},
  journal = {arXiv preprint},
  note    = {arXiv:2601.00758},
  year    = {2026}
}

@article{AiDingLiuYang2026,
  author  = {Ai, Jiangdong and Ding, Laihao and Liu, Hong and Yang, Haotian},
  title   = {Tree suspensions and transfer functions for single degree {Tur{\'a}n}
spectra},
  journal = {arXiv preprint},
  note    = {arXiv:2607.06518},
  year    = {2026}
}

@article{HalfpapLemonsPalmer2025,
  author  = {Halfpap, Anastasia and Lemons, Nathan and Palmer, Cory},
  title   = {Positive co-degree density of hypergraphs},
  journal = {J. Graph Theory},
  volume  = {110},
  number  = {2},
  pages   = {209--222},
  year    = {2025}
}

@incollection{Keevash2011,
  author    = {Keevash, Peter},
  title     = {Hypergraph {Tur{\'a}n} problems},
  booktitle = {Surveys in Combinatorics 2011},
  series    = {London Math. Soc. Lecture Note Ser.},
  volume    = {392},
  pages     = {83--139},
  publisher = {Cambridge Univ. Press},
  address   = {Cambridge},
  year      = {2011}
}

@article{KeevashZhao2007,
  author  = {Keevash, Peter and Zhao, Yi},
  title   = {Codegree problems for projective geometries},
  journal = {J. Combin. Theory Ser. B},
  volume  = {97},
  number  = {6},
  pages   = {919--928},
  year    = {2007},
  doi     = {10.1016/j.jctb.2007.01.004}
}

@article{LoMarkstrom2014,
  author  = {Lo, Allan and Markstr{\"o}m, Klas},
  title   = {{$\ell$}-degree {Tur{\'a}n} density},
  journal = {SIAM J. Discrete Math.},
  volume  = {28},
  number  = {3},
  pages   = {1214--1225},
  year    = {2014}
}

@article{MubayiPikhurko2008,
  author  = {Mubayi, Dhruv and Pikhurko, Oleg},
  title   = {Constructions of non-principal families in extremal hypergraph theory},
  journal = {Discrete Math.},
  volume  = {308},
  number  = {19},
  pages   = {4430--4434},
  year    = {2008}
}

@article{Pikhurko2012,
  author   = {Pikhurko, Oleg},
  title    = {On possible {T}ur{\'a}n densities},
  journal  = {Israel J. Math.},
  volume   = {201},
  number   = {1},
  pages    = {415--454},
  year     = {2014},
  doi      = {10.1007/s11856-014-0031-5}
}

@article{BS:84,
  author   = {Brown, W. G. and Simonovits, M.},
  title    = {Digraph extremal problems, hypergraph extremal problems, and the densities of graph structures},
  journal  = {Discrete Math.},
  volume   = {48},
  number   = {2--3},
  pages    = {147--162},
  year     = {1984},
  doi      = {10.1016/0012-365X(84)90178-X}
}

@article{MubayiRodl2002,
  author  = {Mubayi, Dhruv and R{\"o}dl, Vojt{\v{e}}ch},
  title   = {On the {Tur{\'a}n} number of triple systems},
  journal = {J. Combin. Theory Ser. A},
  volume  = {100},
  number  = {1},
  pages   = {136--152},
  year    = {2002}
}

@article{MubayiZhao2007,
  author  = {Mubayi, Dhruv and Zhao, Yi},
  title   = {Co-degree density of hypergraphs},
  journal = {J. Combin. Theory Ser. A},
  volume  = {114},
  number  = {6},
  pages   = {1118--1132},
  year    = {2007}
}

@article{Reiher2020,
  author  = {Reiher, Christian},
  title   = {Extremal problems in uniformly dense hypergraphs},
  journal = {European J. Combin.},
  volume  = {88},
  pages   = {103117},
  year    = {2020}
}

@article{Sidorenko1995,
  author  = {Sidorenko, Alexander},
  title   = {What we know and what we do not know about {Tur{\'a}n} numbers},
  journal = {Graphs Combin.},
  volume  = {11},
  number  = {2},
  pages   = {179--199},
  year    = {1995}
}

\end{document}